\documentclass{amsart}

\usepackage{graphicx}

\usepackage{amsmath,amsfonts,amsthm,amssymb,amscd}
\usepackage{hyperref}

\usepackage[alphabetic]{amsrefs}

\newcommand*{\textlabel}[2]{%
  \edef\@currentlabel{#1}
  \phantomsection
  #1\label{#2}
}
\makeatother

\newtheorem{theorem}{Theorem}[section]
\newtheorem{lemma}[theorem]{Lemma}
\newtheorem{proposition}[theorem]{Proposition}
\newtheorem{corollary}[theorem]{Corollary}
\newtheorem*{claim}{Claim}

\theoremstyle{definition}
\newtheorem{definition}[theorem]{Definition}

\newtheorem*{ack}{Acknowledgements}
\newtheorem{remark}[theorem]{Remark}
\newtheorem{question}[theorem]{Question}

\numberwithin{equation}{section}

\newcommand {\vab}[1]{\mathcal{F}_{#1}}
\newcommand {\gd}[2]{\mathrm{gd}_{#1}{#2}}
\newcommand {\cd}[2]{\mathrm{cd}_{#1}{#2}}

\newcommand {\ucd}[1]{\underline{\mathrm{cd}}{#1}}

\newcommand {\ee}[2]{E_{#1}{#2}}

\newcommand {\mc}[1]{\mathcal{#1}}
\newcommand {\cze}{\mathrm{CAT}(0)}

\newcommand{\ue}{\underline{E}}

\newcommand{\uue}{ \underline{\underline{E}}}
\newcommand{\conc}{\mathrm{(C)}}

\begin{document}

\title[Bredon dimension for CAT(0) groups]{Bredon cohomological dimension for virtually abelian stabilisers for CAT(0) groups}

\author[T. Prytu{\l}a]{Tomasz Prytu{\l}a}

\address{Max Planck Institute for Mathematics, Vivatsgasse 7, 53111 Bonn, Germany}

\email{prytula@mpim-bonn.mpg.de}



\begin{abstract}Given a discrete group $G$, for any integer $r\geqslant0$ we consider the family of all virtually abelian subgroups of $G$ of rank at most $r$. We give an upper bound for the Bredon cohomological dimension of $G$ for this family for a certain class of groups acting on $\cze$ spaces. This covers the case of Coxeter groups, Right-angled Artin groups, fundamental groups of special cube complexes and graph products of finite groups.
Our construction partially answers a question of J.-F.\ Lafont.
\end{abstract}

\maketitle

\section{Introduction}\label{sec:intro}

Given a group $G$ and a family of its subgroups $\mc{F}$, a classifying space $E_{\mc{F}}G$ is a terminal object in the homotopy category of $G$--CW--complexes with stabilisers in $\mc{F}$. Such a classifying space always exists and any two models for $E_{\mc{F}}G$ are $G$--homotopy equivalent. The two most prominent families are that of all finite and of all virtually cyclic subgroups of $G$, where the corresponding classifying spaces are commonly denoted by $\ue G$ and $\uue G$ respectively. These spaces appear in the formulation of Baum-Connes and Farrell-Jones conjecture respectively, and they have been intensively studied over the last decade \cite{Lusurv}. A particular emphasis has been put on constructing simple models for $E_{\mc{F}}G$, meaning of the possibly small dimension and cell structure.
The minimal dimension of a model for $E_{\mc{F}}G$ is called the \emph{geometric dimension of $G$ for the family $\mathcal{F}$} and is denoted by $\gd{\mc{F}}{G}$. There is an algebraic counterpart of geometric dimension, which is called the \emph{Bredon cohomological dimension} and is denoted by $\cd{\mc{F}}{G}$. These dimensions satisfy the following inequality
\[\cd{\mc{F}}{G} \leqslant \gd{\mc{F}}{G} \leqslant \mathrm{max} \{3,\cd{\mc{F}}{G}\}.\]
Therefore to show the existence of a model for $E_{\mc{F}}G$ of a given dimension, it is enough to bound the Bredon cohomological dimension, which sometimes turns out to be an easier task. Finite-dimensional models for the families of finite and virtually cyclic groups have been constructed for many classes of groups, particularly for numerous non-positively curved groups \cite{DegPe,Lusurv,Lucze, OsaPry}. 

More recently, as a natural generalisation of families of finite and of virtually cyclic subgroups, the family of virtually abelian subgroups has been considered. For any $r>0$ let $\vab{r}$ denote the family of all subgroups of $G$ that are finitely generated virtually abelian of rank at most $r$. Constructing finite-dimensional models for $\ee{\vab{r}}{G}$ for $r>1$ is substantially harder. To date, very few constructions have been performed, and all of them use a very simple structure of this family. In [OP16] together with D.\ Osajda we have constructed finite-dimensional models for $\ee{\vab{2}}{G}$ where $G$ is a group acting properly on either a systolic complex or a $\cze$ space with no flats of dimension higher than $2$. In \cite{CMNP} there is a construction of a finite-dimensional model for $\ee{\vab{r}}{G}$ where $G$ is finitely generated abelian group.

In the current paper, building on the above constructions and the results of Degrijse-Petrosyan \cite{DegPe}, we give an upper bound for the dimension $\cd{\vab{r}}{G}$ for a certain class of groups acting on $\cze$ spaces.
Recall that a $\cze$ space is a simply connected geodesic metric space of non-positive metric curvature (see \cite{BH} for a detailed treatment). 

\begin{theorem}\label{thm:intromain}Let $G$ be a group acting properly by semi-simple isometries on a complete proper $\mathrm{CAT}(0)$ space of topological dimension $n$. Suppose additionally that $G$ satisfies condition $\conc$. Then for any $0 \leqslant r \leqslant n$ we have 
\[\cd{\vab{r}}G \leqslant n+r+1.\]
\end{theorem}

The condition $\conc$ is a higher-rank analogue of L\"uck's condition $\conc$ for cyclic groups. Roughly speaking, it says that commensurators of elements of $\vab{r}$ can be approximated by their normalisers.  
This condition is satisfied for example by linear groups.
The following is our main application of Theorem~\ref{thm:intromain}.

\begin{corollary}\label{cor:introcor}Let $G$ be as in Theorem~\ref{thm:intromain} and suppose it is linear over $\mathbb{Z}$. Then for any $0 \leqslant r \leqslant n$ we have 
\[\cd{\vab{r}}G \leqslant n+r+1.\]
\end{corollary}

Among groups satisfying the above assumptions are Coxeter groups, Right-angled Artin groups, the so-called special groups (fundamental groups of special cube complexes) and graph products of finite groups.

Unlike L\"{u}ck's condition, our condition $\conc$ is not satisfied by all $\cze$ groups \cite{LM}. In the forthcoming paper with J.\ Huang \cite{HuPr}, we show condition $\conc$ for several classes of $\cze$ groups, including $\cze$ cubical groups. It is conceivable that it holds for more general classes of $\cze$ groups, such as biautomatic $\cze$ groups.

We remark that Theorem~\ref{thm:intromain} partially answers a question of J.-F.\ Lafont \cite[Problem 46.7]{Gui}. Namely, Lafont inquired whether a 'good structure' of (virtually) abelian subgroups of a given group $G$ enables one to construct models for $\ee{\vab{r}}{G}$ more easily than models for $\ee{\vab{1}}{G}=\uue G$. While we use the specific structure of abelian subgroups to bound $\cd{\vab{r}}{G}$ (and thus we obtain a model for $\ee{\vab{r}}{G}$), our construction is an inductive procedure that takes into account abelian subgroups of all ranks lower than $r$, and therefore is by no means easier. Our main tool that is used to study virtually abelian groups of $G$ is the so-called Flat Torus Theorem, which is one of the fundamental theorems in the theory of $\cze$ groups. 

As a step in our construction we first prove Theorem~\ref{thm:intromain} in a very special case where $G$ is virtually abelian, thus generalising the main result of \cite{CMNP}. 

\begin{proposition}\label{prop:introabab}Let $G$ be a finitely generated virtually abelian group of rank $n \geq 0$. Then for any $0 \leqslant r \leqslant n-1$ we have \[\mathrm{gd}_{\vab{r}}G \leqslant n+r.\]
\end{proposition}
Note that in this case we obtain a slightly better dimension bound, and also our construction here is purely topological (i.e.,\ we give a concrete construction of a model for $\ee{\vab{r}}{G}$).

We would like to point out that an even more general family of \emph{polycyclic} stabilisers is considered in \cite{VM}. In this case the rank is replaced by the \emph{Hirsch length}, and the inductive approach is similar to ours and the one in \cite{CMNP}. Since a $\cze$ group is polycyclic if and only if it is virtually abelian \cite[Theorem~II.7.16]{BH}, our construction overlaps with \cite{VM} only in Proposition~\ref{prop:introabab}.

\begin{ack}I would like to thank Victor Moreno for useful discussions, and particularly for pointing out that by using Lemma~\ref{lem:modelunionoffam} one can obtain much better dimension bounds. I also thank Ian Leary, Ashot Minasyan and Nansen Petrosyan for helpful discussions.

The author was supported by the EPSRC First Grant EP/N033787/1 and by (Polish) Narodowe Centrum Nauki, grant no.~UMO-2015/18/M/ST1/00050.
\end{ack}

\section{Preliminaries}\label{sec:prelim}

In this section we recall some basic notions and we present an inductive procedure of constructing the classifying space $\ee{\vab{r}}{G}$. We also briefly discuss the Bredon cohomological counterpart of that procedure, which we use in Section~\ref{sec:geomdimcat0} to prove Theorem~\ref{thm:mainthm}. Our exposition focuses on the topological construction, since we believe that it gives the reader more insight and is somewhat easier to picture.

\subsection{Classifying spaces for families of subgroups}\label{subsec:classspace}

Let $G$ be a discrete group. A \emph{family of subgroups} $\mathcal{F}$ of $G$ is a collection of subgroups that is closed under taking subgroups and under conjugation by elements of $G$. 

\begin{definition}A $G$--CW--complex $X$ is a \emph{model for the classifying space} $E_{\mathcal{F}}G$ if for any subgroup $H \subset G$ the fixed point set $X^H$ is contractible if $H\in \mathcal{F}$ and empty otherwise.

The \emph{geometric dimension for the family $\mathcal{F}$} denoted by $\mathrm{gd}_{\mathcal F}G$ is the minimal dimension of a model for the classifying space $E_{\mathcal{F}}G$. 
\end{definition}
The classifying space always exists and has the following universal property. Let $Y$ be any $G$--space with stabilisers in $\mathcal{F}$. Then there exists a $G$--map $Y \to E_{\mathcal{F}}G$ which is unique up to $G$--homotopy. In particular, any two models for $E_{\mathcal{F}}G$ are $G$--homotopy equivalent (see \cite{Lusurv}). However, general constructions always yield infinite dimensional CW--complexes. The standard (by now) method to construct finite-dimensional models for $E_{\mathcal{F}}G$ is the construction of L\"{u}ck and Weiermann \cite{LuWe}, which we will now briefly recall. 

We need a piece of notation. Given a group $G$, a family of its subgroups $\mathcal{F}$ and a subgroup $K \subset G$, let $\mc{F} \cap K$ denote the collection of subgroups of $\mc{F}$ that belong to $K$. One easily verifies that $\mc{F} \cap K$ is a family of subgroups of $K$.

Now let $\mathcal{F}$ and $\mathcal{G}$ be two families of subgroups of $G$ such that $\mathcal{F} \subset \mathcal{G}$. Suppose that we have an equivalence relation $\sim$ on the set of subgroups $\mathcal{G} \setminus \mathcal{F}$ which satisfies the following two conditions:

\begin{itemize}
	\item for any $H, H' \in \mathcal{G} \setminus \mathcal{F}$ if $H' \subset H$ then $H \sim H'$,
	\item for every $g \in G$ we have $H \sim H'$ if and only if $g^{-1}Hg \sim g^{-1}H'g$.
\end{itemize}
Such a relation is called a \emph{strong equivalence relation}.
Let $[H]$ denote the equivalence class of $H$ and let $[\mathcal{G} \setminus \mathcal{F}]$ denote the set of equivalence classes of elements of $\mathcal{G} \setminus \mathcal{F}$. Now given $[H] \in [\mathcal{G} \setminus \mathcal{F}]$ let $N_G[H]$ be a subgroup of $G$ defined as follows
\[N_G[H]=\{ g \in G \mid g^{-1}Hg \sim H\}.\]
Define the family of subgroups $\mathcal{G}[H]$ of $N_G[H]$ by
\[\mathcal{G}[H] = \{H' \in \mathcal{G} \setminus \mathcal{F} \mid H' \sim H\} \cup (\mathcal{F} \cap N_G[H]).\]
The following is one of the main theorems of \cite{LuWe}.

\begin{theorem}{\cite[Theorem 2.3]{LuWe}} Let $I$ be the complete set of representatives of $G$--conjugacy classes of elements of $[\mathcal{G} \setminus \mathcal{F}]$. Suppose we have a model for $E_{\mathcal{F}}G$ and for any $[H] \in I$ we have models for $E_{\mathcal{F} \cap N_G[H]}N_G[H]$ and $E_{\mathcal{G}[H]}N_G[H]$.
Let $X$ be given by a cellular $G$--pushout
\[
\begin{CD}
\coprod_{[H] \in I}  G \times_{N_G[H]} E_{\mathcal{F} \cap N_G[H]} N_G[H]     @>i>>  E_{\mathcal{F}}G \\
@VV{\coprod_{[H] \in I} \mathrm{id}_G \times_{N_G[H]} f_{[H]}}V         @VVV\\
\coprod_{[H] \in I}  G \times_{N_G[H]} E_{\mathcal{G}[H]} N_G[H]     @>{\phantom{\text{label}}}>>   X,
\end{CD}
\]
where $i$ is a $G$--$\mathrm{CW}$--inclusion and every $f_{[H]}$ is a cellular $N_G[H]$--map. Then $X$ is a model for $E_{\mathcal{G}}G$.
\end{theorem}

The existence of maps $i$ and $f_{[H]}$ follows from the universal properties of appropriate classifying spaces. Moreover, if the map $i$ fails to be injective, we can replace it with the inclusion into the mapping cylinder. Note that for a map $f \colon Y \to Z$ the dimension of the mapping cylinder is equal to $\mathrm{max}\{\mathrm{dim}Y +1, \mathrm{dim}Z\}$. Also note that by restricting the $G$--action to $N_G[H]$, a model for $E_{\mathcal{F}}G$ may be seen as a model for $E_{\mathcal{F} \cap N_G[H]} N_G[H]$. These considerations lead to the following corollary which gives us control on the dimension of the obtained model.

\begin{corollary}\cite[Remark~2.5]{LuWe}\label{cor:dimmodel}
With the above notation we have the following inequality 
\[\mathrm{gd}_{\mathcal G}G \leqslant  \mathrm{max}\big\{\mathrm{gd}_{\mathcal F}G+1, \underset{{[H] \in I}}{\mathrm{max}} \{\mathrm{gd}_{\mathcal G[H]}N_G[H]\} \big\}.\]
\end{corollary}

We finish this section with a lemma that relates geometric dimension of a group to geometric dimension of its finitely generated subgroups.

\begin{lemma}\cite[Theorem~4.3]{LuWe}\label{lem:modellimit} Let $G$ be a  group and $\mathcal{F}$ a family of subgroups of $G$ such that any $F \in \mathcal{F}$ is contained in some finitely generated subgroup of $G$. Suppose that there exists an integer $d>0$ such that for every finitely generated subgroup $K$ of $G$ we have $\mathrm{gd}_{\mathcal{F} \cap K}K \leqslant d$. Then $\mathrm{gd}_{\mathcal{F}}G \leqslant d+1$.
\end{lemma}

\subsection{Virtually abelian stabilisers}\label{subsec:vabstab}

For an integer $r>0$ let $\vab{r}$ denote the family of all subgroups of $G$ that are finitely generated virtually abelian of rank at most $r$, i.e., we have $ H \in \vab{r}$ if and only if $H$ contains a finite-index subgroup isomorphic to $\mathbb{Z}^n$ with $n \leqslant r$. Notice that we have the following chain of inclusions
\[\vab{0} \subset \vab{1} \subset \vab{2} \subset \ldots \subset \vab{r} \subset \ldots. \]
Also note that $\vab{0}$ is the family of all finite subgroups, and  $\vab{1}$ is the family of all virtually cyclic subgroups.

We will now describe the construction outlined in the previous section, applied to the inclusion of families $\vab{r-1} \subset \vab{r}$. Define an equivalence relation on $\vab{r} \setminus \vab{r-1}$ by
\[H \sim H' \text{ if and only if } \mathrm{rk}(H \cap H') =r.\]
One easily verifies that it is a strong equivalence relation.
For $[H] \in [\vab{r} \setminus \vab{r-1}]$ the group $N_G[H]$ is defined as \[N_G[H]=\{g \in G \mid \mathrm{rk}(g^{-1}Hg \cap H)=r\}\] and it is called the \emph{commensurator} of $H$ in $G$. The family $\vab{r}[H]$ is given by \[\vab{r}[H] = \{H' \in \vab{r} \setminus \vab{r-1} \mid \mathrm{rk}(H' \cap H)=r \} \cup (\vab{r-1} \cap N_G[H]).\] 

The most difficult part is to construct models for $E_{\vab{r}[H]}N_G[H]$. For this, first observe that the family $\vab{r}[H]$ can be written as the union \[\vab{r}[H] = \mathcal{ALL}[H] \cup (\vab{r-1} \cap N_G[H]),\]
where $\mathcal{ALL}[H]$ denotes the smallest family of subgroups of $G$ that contains all representatives $H'$ of $[H]$. 

We have the following two lemmas.

\begin{lemma}{\cite[Lemma~2.4]{CMNP}}\label{lem:modelunionoffam} Let $G$ be a group and let $\mathcal{F}$ and $\mathcal{G}$ be two families of subgroups. Then we have
\[\mathrm{gd}_{\mathcal{F} \cup \mc{G}}{G} \leqslant \mathrm{max} \{ \gd{\mc{F}}{G} ,\gd{\mc{G}}{G}, \gd{\mc{F} \cap \mc{G}}{G}+1  \}.\]
\end{lemma} 

\begin{lemma}\cite[Proposition~5.1(i)]{LuWe}\label{lem:passingtosmaller} Let $G$ be a group and let $\mathcal{F}$ and $\mathcal{G}$ be two families of subgroups such that $\mc{F} \subset \mc{G}$. Suppose for every $H \in \mc{G}$ we have $\gd{\mc{F}\cap H}{H} \leqslant d$. Then \[\gd{\mc{F}}{G} \leqslant \gd{\mc{G}}{G}+d.\]\end{lemma}

We are ready now to prove the special case of the main theorem. The proof of the main theorem will follow the same lines, using this special case as one of the steps.

\begin{proposition}\label{prop:dimvab}Let $G$ be virtually $\mathbb{Z}^n$ for some $n \geqslant 0$. Then for any $0\leqslant r <n$ we have \[\gd{\vab{r}}{G} \leqslant n+r.\]
\end{proposition}

\begin{proof} We proceed by induction on $r$. For $r=0$ we have
$\gd{\vab{r}}{G} =n$ \cite[Theorem~5.13]{LuWe}. Assume that $\gd{\vab{r-1}}{G} \leqslant n+r-1$ and consider the inclusion of families $\vab{r-1} \subset \vab{r}$. By Corollary~\ref{cor:dimmodel} we have
\[\mathrm{gd}_{\vab{r}}G \leqslant  \mathrm{max}\big\{\mathrm{gd}_{\vab{r-1} }G+1, \underset{{[H] \in I}}{\mathrm{max}} \{\mathrm{gd}_{\vab{r}[H]}N_G[H]\} \big\},\]
so in order to get the claim we need to show that $\mathrm{max}_{[H] \in I} \{\mathrm{gd}_{\mathcal G[H]}N_G[H] \} \leqslant n+r$.  Take $[H] \in I$ and consider the commensurator $N_G[H]$. One easily sees that $N_G[H]$ is virtually $\mathbb{Z}^n$. By Lemma~\ref{lem:modelunionoffam} we have
\begin{align}
\begin{split}
\gd{\vab{r}[H]}{N_G[H]} \leqslant    \mathrm{max}\{&\gd{\mathcal{ALL}[H]}{N_G[H]},\gd{\vab{r-1}\cap N_G[H]}{N_G[H]},\\& \gd{ALL[H]\cap \vab{r-1} }{N_G[H]}+1 \}
\end{split}
\end{align}
First note that since $N_G[H] \subset G$, by inductive assumption we get that 
\begin{equation}\gd{\vab{r-1}\cap N_G[H]}{N_G[H]} \leqslant n+r-1.
\end{equation}
Next, we will bound $\gd{\mathcal{ALL}[H]}{N_G[H]}$. For this we need the following claim. 
\begin{claim} There exists $H'$ with $H \sim H'$ such that $H'$ is normal in $N_G[H]$.
\end{claim}
The group $N_G[H]$ contains $\mathbb{Z}^n$ as a finite-index subgroup. We can assume that $H \subset \mathbb{Z}^n$ (be replacing $H$ with $H \cap \mathbb{Z}^n$ if necessary).
Now since $\mathbb{Z}^n$ and $H$ commute and since $\mathbb{Z}^n$ has finite index in $N_G[H]$ it follows that there are only finitely many distinct $N_G[H]$--conjugates of $H$. Moreover, by definition of $N_G[H]$ any two such conjugates are equivalent, i.e.,\ their intersection has rank $r$. Define $H'$ to be the intersection of all $N_G[H]$--conjugates of $H$. One easily sees that $H' \sim H$ and that $H'$ is normal in $N_G[H]$. This finishes the proof of the claim.\smallskip

Now consider the quotient $N_G[H]/H'$. We have that any model for $E_{\vab{0}}N_G[H]/H'$ is a model for $E_{\mathcal{ALL}[H]}N_G[H]$, where the action is given by the projection \[N_G[H] \to N_G[H]/H'.\] This is because a subgroup $F \subset N_G[H]$ belongs to the family $\mathcal{ALL}[H]$ if and only if its image under $N_G[H] \to N_G[H]/H'$ is finite.
The group $N_G[H]/H'$ contains $\mathbb{Z}^{n-r}$ as a finite-index subgroup, and thus by \cite[Theorem 5.13]{LuWe} there exists a $(n-r)$--dimensional model for $E_{\vab{0}}N_G[H]/H'$. Consequently, we have \begin{equation}\label{eq:lemvab:all}\gd{\mathcal{ALL}[H]}{N_G[H]} \leqslant n-r. \end{equation}
It remains to bound $\gd{\mathcal{ALL}[H]\cap \vab{r-1}}{N_G[H]}$. We would like to use Lemma~\ref{lem:passingtosmaller} applied to the families \[\mathcal{ALL}[H]\cap \vab{r-1} \subset \mathcal{ALL}[H].\] For this, we have to estimate $\gd{\vab{r-1}}{F}$ for every $F \in \mathcal{ALL}[H]$. Note that if $F$ has rank  at most $r-1$ then $\gd{\vab{r-1}}{F}=0$, as in this case the one-point space is a $0$--dimensional model for $\ee{\vab{r-1}}{F}$. 
If $F$ has rank $r$ then by the inductive assumption we have $\gd{\vab{r-1}}{F} \leqslant r+r-1$.  

Thus by Lemma~\ref{lem:passingtosmaller} and by \eqref{eq:lemvab:all} we conclude that 
\begin{equation*}\gd{\mathcal{ALL}[H]\cap \vab{r-1}}{N_G[H]} \leqslant (n-r) + (r+r-1)= n+r-1.\qedhere
\end{equation*}
\end{proof}

In general, the structure of the commensurator $N_G[H]$ may be complicated. However, if $G$ satisfies the following condition, using Lemma~\ref{lem:modellimit} one can approximate $N_G[H]$ by normalisers of subgroups $H'$ with $H' \sim H$ (which are usually easier to deal with). This condition is the higher rank analogue of L\"{u}ck's condition $\conc$ \cite[Condition~4.1]{Lucze}.

\begin{definition}[Condition $\conc$]\label{cond:c} We say that a group $G$ satisfies condition $\conc$ if for every $H \in \vab{r}$ and for every finitely generated subgroup $K \subset N_G[H]$ there exists a subgroup $H' \in \vab{r}$ with $H \sim H'$ such that $\langle K , H \rangle \subset N_G(H')$.
\end{definition}

Notice that the Claim in the proof of Proposition~\ref{prop:dimvab} states that a virtually abelian groups satisfy (a strong form of) condition~\hyperref[cond:c]{$\conc$}.

\subsection{Bredon cohomological dimension $\mathrm{cd}_{\mathcal{F}}G$.}\label{subsec:bredondim}

Bredon cohomology is an equivariant cohomology theory that  is well suited for studying classifying spaces for families of subgroups. We will not recall general theory and refer the reader to \cite{luckbook}. Roughly, in comparison with usual group cohomology where one considers $G$--modules, in Bredon cohomology one considers $\mathcal{O}_{\mc{F}}G$--modules, where $\mathcal{O}_{\mc{F}}G$ is the \emph{orbit category} of $G$ with respect to the family $\mathcal{F}$. The category of $\mathcal{O}_{\mc{F}}G$--modules is an abelian category with enough projectives, and thus one can `do' the homological algebra there. Consequently, there is a notion of \emph{Bredon cohomological dimension for the family} $\mc{F}$ which  is defined to be 
\[\cd{\mc{F}} G = \sup\{ n \in \mathbb{N} \ | \ \mathrm{H}^n_{\mathcal{F}}(G,M)\neq 0 \text{ for some } \mathcal{O}_{\mc{F}}G\text{--module } M  \}, \]
where $\mathrm{H}^n_{\mathcal{F}}(G,M)$ denotes the $n$--th Bredon cohomology group with coefficients in $M$. As mentioned in the Introduction, Bredon cohomological dimension is related to geometric dimension by the following inequality \cite[Theorem~0.1]{LuMe}:
\[\cd{\mc{F}}{G} \leqslant \gd{\mc{F}}{G} \leqslant \mathrm{max} \{3,\cd{\mc{F}}{G}\}.\]

The method of constructing classifying spaces presented in Subsections~\ref{subsec:classspace} and \ref{subsec:vabstab} has its cohomological counterpart, which, rather than a model for $E_{\mathcal{F}}G$, produces a resolution of the trivial $\mathcal{O}_{\mc{F}}G$--module $\underline{\mathbb{Z}}$ by projective $\mathcal{O}_{\mc{F}}G$--modules. In particular, it can be used to obtain an upper bound for $\cd{\mc{F}}{G}$.  For a detailed account of the construction see \cite[Section~8.2]{Degthesis} or \cite[Section~3]{DegPe}. All the steps of the construction in Subsections~\ref{subsec:classspace} and \ref{subsec:vabstab} have their cohomological counterparts, which we sum up in the following remark.

\begin{remark}\label{rem:cdvsgd}Corollary~\ref{cor:dimmodel}, Lemma~\ref{lem:modellimit}, Lemma~\ref{lem:modelunionoffam} and Lemma~\ref{lem:passingtosmaller} all remain true if one replaces $\gd{\mc{F}}{G}$ by $\cd{\mc{F}}{G}$. For the proofs see respectively \cite[Corollary~8.2.9]{Degthesis}, \cite[Corollary~8.2.5(i)]{Degthesis}, \cite[Lemma~4.20]{VM} and \cite[Corollary~8.2.3]{Degthesis}.

When we refer to a `cohomological dimension version' of one of the above, we will add a suffix \textlabel{(cd)}{cond:cd} to the respective reference. For example, if we want to refer to Lemma~\ref{lem:modellimit} applied to Bredon cohomological dimension we will write Lemma~\ref{lem:modellimit}.\hyperref[cond:cd]{(cd)}.
\end{remark}

\section{Bredon cohomological dimension for $\mathrm{CAT}(0)$ groups}\label{sec:geomdimcat0}

A $\cze$ space is a geodesic metric space satisfying the so-called $\cze$ inequality \cite[Definition~II.1.1]{BH}. We make no use of the definition and only use properties of $\cze$ spaces and their isometries. We refer the reader to \cite{BH} for a detailed treatment of $\cze$ geometry and for definitions of various notions presented below. For a definition and some properties of topological dimension we refer the reader to \cite{DegPe} or \cite{Lucze}.

In this section we prove the main theorem of this paper. 

\begin{theorem}\label{thm:mainthm}Let $G$ be a group acting properly by semi-simple isometries on a complete proper $\mathrm{CAT}(0)$ space of topological dimension $n$. Suppose additionally that $G$ satisfies condition~\hyperref[cond:c]{$\conc$}. Then for any $0 \leqslant r \leqslant n$ we have 
\[\cd{\vab{r}}G \leqslant n+r+1.\]
\end{theorem}
Note that if $G$ acts cocompactly then it automatically acts by semi-simple isometries. The following is the key result needed in our approach.

\begin{lemma}{\cite[Corollary~1]{DegPe}}\label{lem:dietnans} Suppose $G$ acts properly by semi-simple isometries on a complete proper $\mathrm{CAT}(0)$ space $X$ of topological dimension $n$. Then we have $\ucd{G}=\cd{\vab{0}}{G}\leqslant n$. 
\end{lemma}

Observe that by the Fixed Point Theorem for $\cze$ spaces \cite[Corollary~II.2.8]{BH}, in the above situation, every finite subgroup of $G$ has a fixed point in $X$ and every fixed point set is contractible. However, the space $X$ is not a model for $\ue G =\ee{\vab{0}}{G}$ since in general it is not a $G$--CW--complex. By taking the nerve of an appropriate open cover of $X$ one obtains a $G$--CW--complex $Y$ homotopy equivalent to $X$, however, in case $\mathrm{dim}(X)=2$ it could happen that $\mathrm{dim}(Y)=3$ (see \cite{DegPe, Lucze}). Because of this, in our construction we use Bredon cohomological dimension, rather than the geometric dimension (see Remark~\ref{rem:cdvsgd}). We remark that our construction is geometric in spirit, and its key point is the use of the Flat Torus Theorem. 
\medskip 

From now on let $G$ be as in Theorem~\ref{thm:mainthm}, and let $X$ denote the $n$--dimensional $\cze$ space acted upon by $G$. The following lemma is the crucial geometric ingredient in the proof of the main theorem.

\begin{lemma}\label{lem:ftt}Suppose that $0<r\leqslant n$ and consider a subgroup $H \cong \mathbb{Z}^r \subset G$. Then the quotient $N_G(H)/H$ acts properly by semi-simple isometries on a complete proper $\cze$ space of dimension at most $n-r$.
\end{lemma}
\begin{proof}
 By the Flat Torus Theorem \cite[Theorem~II.7.1]{BH} there is an $N_G(H)$--invariant convex subspace $Z \subset X$, such that $Z$ is isometric to the product $Y \times \mathbb{R}^r$ and the action of $H \subset N_G(H) $ on $Z$ is trivial on the $Y$ factor and free by translations on the $\mathbb{R}^r$ factor. It follows from \cite[Proposition~II.6.9 and Proposition~II.6.10(4)]{BH} that $N_G(H)/H$ acts properly by semi-simple isometries on $Y$. Since $Y$ is a convex subspace of $X$ we get that $Y$ with the induced metric is a $\cze$ space. Moreover, we have $\mathrm{dim}(Y \times \mathbb{R}^r) \leqslant \mathrm{dim}(X)=n$ and therefore we conclude that $\mathrm{dim}(Y) \leqslant n-r$.
\end{proof}

\begin{lemma}\label{lem:cat0quotientandsmaller}Let $[H] \in  [\vab{r} \setminus \vab{r-1}]$ for some $r\geq 1$. Then we have \[\cd{\mathcal{ALL}[H]}N_G[H] \leqslant n-r+1 \text{ and } \cd{\mathcal{ALL}[H] \cap \vab{r-1}}N_G[H] \leqslant n+r.\]
\end{lemma}

\begin{proof} For the first inequality, in light of Lemma~\ref{lem:modellimit}.\hyperref[cond:cd]{(cd)} (see Remark~\ref{rem:cdvsgd}) it is enough to show that for every finitely generated subgroup $K \subset N_G[H]$ with $H\subset K$ we have $\cd{\mathcal{ALL}[H] \cap K }K \leqslant n-r$.  Let $K \subset N_G[H]$ be such a subgroup. Since $G$ satisfies Condition~\hyperref[cond:c]{$\conc$}, we get that there exists a  subgroup $H' \cong \mathbb{Z}^r$ in $G$ such that $H' \sim H$ and $\langle K, H \rangle$ is a subgroup of the normaliser $N_G(H')$. Since by Shapiro's lemma (see \cite{DegPe}) we have $ \cd{\mathcal{ALL}[H] \cap K }K \leqslant  \cd{\mathcal{ALL}[H] \cap \langle K, H \rangle } \langle K, H \rangle \leqslant \cd{\mathcal{ALL}[H] \cap N_G(H')}N_G(H')$, it is enough to show that \[\cd{\mathcal{ALL}[H] \cap N_G(H')}N_G(H') \leqslant n-r.\] 

For this, we will use the fact that \[\cd{\mathcal{ALL}[H] \cap N_G(H')}N_G(H') \leqslant \cd{\vab{0}} {N_G(H')/H'}\]  \cite[Corollary 8.2.4]{Degthesis}, which is the cohomological analogue of the fact that any model for $E_{\vab{0}}N_G(H')/H'$ is a model for $E_{\mathcal{ALL}[H] \cap N_G(H')}N_G(H')$ via the projection $N_G(H') \to N_G(H')/H'$.

By Lemma~\ref{lem:ftt} the quotient $N_G(H')/H'$ acts properly by semi-simple isometries on a complete proper $\cze$ space of dimension at most $n-r$. Thus applying Lemma~\ref{lem:dietnans} we get that $\cd{\vab{0}}{N_G(H')/H'}\leqslant n-r$. This finishes the proof of the first inequality.\medskip

For the second inequality we would like to use Lemma~\ref{lem:passingtosmaller}.\hyperref[cond:cd]{(cd)} applied to the inclusion of families $\mathcal{ALL}[H] \cap \vab{r-1}\subset \mathcal{ALL}[H]$.  For this we need to bound $\cd{\vab{r-1}}{F}$ for every $F \in \mathcal{ALL}[H]$. If $F$ has rank  at most $r-1$ then $\cd{\vab{r-1}}{F}=0$. If $F$ has rank $r$ then by Proposition~\ref{prop:dimvab} we have $\gd{\vab{r-1}}{F} \leqslant r+r-1$ and thus also $\cd{\vab{r-1}}{F} \leqslant r+r-1$. By the first inequality and Lemma~\ref{lem:passingtosmaller}.\hyperref[cond:cd]{(cd)} we get that   \[\cd{\mathcal{ALL}[H] \cap \vab{r-1}}N_G[H] \leqslant (n-r+1)+(r+r-1)=n+r.\qedhere\]
\end{proof}

\begin{proof}[Proof of Theorem~\ref{thm:mainthm}]
We proceed by induction on $r$. For $r=0$ by Lemma~\ref{lem:dietnans} we have $\cd{\vab{0}}{G}\leqslant n \leqslant n+1$. Now assume inductively that we have 
\[\cd{\vab{r-1}}{G}\leqslant n+r,\] and consider the inclusion of families $\vab{r-1} \subset \vab{r}$. We will show that $\cd{\vab{r}}{G}\leqslant n+r+1$. By Corollary~\ref{cor:dimmodel}.\hyperref[cond:cd]{(cd)} it is enough to show that for any $[H] \in [\vab{r} \setminus \vab{r-1}]$ we have $\mathrm{gd}_{\vab{r}[H]}N_G[H]\leqslant n+r+1$. By Lemma~\ref{lem:modelunionoffam}.\hyperref[cond:cd]{(cd)} we have 
\begin{align}\label{eq:cat0union}
\begin{split}
\cd{\vab{r}[H]}{N_G[H]} \leqslant    \mathrm{max}\{&\cd{\mathcal{ALL}[H]}{N_G[H]},\cd{\vab{r-1}\cap N_G[H]}{N_G[H]},\\& \cd{ALL[H]\cap \vab{r-1} }{N_G[H]}+1 \}
\end{split}
\end{align}
By inductive assumption we have $\cd{\vab{r-1}\cap N_G[H]}{N_G[H]} \leqslant n+r$. By Lemma~\ref{lem:cat0quotientandsmaller} we get $\cd{\mathcal{ALL}[H]}{N_G[H]} \leqslant n-r+1$ and $\cd{ALL[H]\cap \vab{r-1} 
}{N_G[H]} \leqslant n+r$. Plugging these inequalities into \eqref{eq:cat0union} we conclude that $\cd{\vab{r}[H]}{N_G[H]} \leqslant n+r+1$.\end{proof}

\begin{remark}In certain cases the dimension bounds in Theorem~\ref{thm:mainthm} can be improved. Namely, as stated in Lemma~\ref{lem:dietnans}, if $r=0$ then $\cd{\vab{0}}{G}\leqslant n$. Also, if $r=1$ then by \cite[Corollary~3(i)]{DegPe} we have $\cd{\vab{0}}{G}\leqslant n+1$. 
\end{remark}

\section{Applications}\label{sec:applications}

In this section we discuss applications of Theorem~\ref{thm:mainthm}. In particular, we derive Corollary~\ref{cor:introcor}.

\begin{lemma}\label{lem:conditionc}Let $G$ be a group which is linear over $\mathbb{Z}$. Then $G$ satisfies Condition~\hyperref[cond:c]{$\conc$}.
\end{lemma}

\begin{proof}Take an arbitrary subgroup $H \in \vab{r}$. By replacing $H$ with its finite-index abelian subgroup if necessary, we can assume that $H$ is abelian. Consider a finitely generated subgroup $K \subset N_G[H]$. Now since $G$ is linear, so is the subgroup $\langle K , H \rangle$. By \cite[Theorem 4.C.5]{Segal} the group $H$ is separable in $\langle K , H \rangle$. By \cite[Corollary~7.(i)]{CKRW} there exists a finite index subgroup $H' \subset H$ which is normal in $\langle K , H \rangle$.\end{proof}

We are ready to prove now Corollary~\ref{cor:introcor}.

\begin{proof}[Proof of Corollary~\ref{cor:introcor}]
The corollary follows immediately from Theorem~\ref{thm:mainthm} and Lemma~\ref{lem:conditionc}.
\end{proof}

Below we present a list of groups to which Corollary~\ref{cor:introcor} applies.
\begin{enumerate}

\item \label{it:1} Let $(W, S)$ be a Coxeter group generated by a finite set $S$. Then $W$ acts properly cocompactly by isometries on Davis complex $\Sigma_W$, which is a complete proper $\cze$ space \cite{Davbook}. It is well known that Coxeter groups are linear (see \cite{Davbook}).

\item \label{it:2} Let $A_L$ be a Right-angled Artin group defined by a finite graph $L$. Then $A_L$ acts freely and properly cocompactly by isometries on the universal cover of the Salvetti complex, which is a complete proper $\cze$ space (cube complex) of dimension equal to the size of the largest clique in $L$. The group $A_L$ is linear (see \cite[Corollary~3.6]{HsuWis}). 

\item \label{it:3} Let $G_L$ be a graph product of finite groups over a finite graph $L$. Then $G_L$ acts properly cocompactly by isometries on a right-angled building, which is  a complete proper $\cze$ space of dimension equal to the size of the largest clique in $L$. The group $G_L$ is linear by \cite[Corollary 3.5]{HsuWis}.

\item \label{it:4} Let $X$ be a compact special cube complex (see \cite{HaWis}). Then $\pi_1(X)$ acts freely and properly cocompactly by isometries on the universal cover $\widetilde{X}$, which is a complete proper $\cze$ cube complex. By \cite[Theorem~1.1]{HaWis} the group $\pi_1(X)$ is linear.
\end{enumerate}

We remark that linearity of groups in (\ref{it:2}) and (\ref{it:3}) is proven by embedding these groups into Coxeter groups. There are many more groups which embed into Coxeter groups. Since all the assumptions of Theorem~\ref{thm:mainthm} pass to subgroups, the Theorem holds for these groups as well. The advantage of classes $A_L$ and $G_L$ above is that, in both cases, one has a naturally assigned $\cze$ space whose dimension depends only on $L$. 

On the other hand, there are $\cze$ groups which do not satisfy condition~~\hyperref[cond:c]{$\conc$} \cite{LM}. In a joint work with J.\ Huang \cite{HuPr}, we show that condition~~\hyperref[cond:c]{$\conc$} is satisfied by several classes of $\cze$ groups, including $\cze$ cubical groups and fundamental groups of non-positively curved Riemannian manifolds. We conclude with the following question. 

\begin{question}Does every biautomatic $\cze$ group satisfy condition~\hyperref[cond:c]{$\conc$}?
\end{question}


\begin{bibdiv}

\begin{biblist}
	
\bib{BH}{book}{
	author={Bridson, M.},
	author={Haefliger, A.},
	title={Metric spaces of non-positive curvature},
	series={Grundlehren der Mathematischen Wissenschaften},
	volume={319},
	publisher={Springer-Verlag, Berlin},
	date={1999},
	pages={xxii+643},
	isbn={3-540-64324-9},
	review={\MR{1744486}},
}


 \bib{CKRW}{article}{
    title     ={On the residual and profinite closures of commensurated subgroups
},
    author    ={Caprace, Pierre-Emmanuel },
    author    ={Kropholler, Peter H. },

    author    ={Reid, Colin D. },
    author    ={Wesolek, Phillip },
    status    ={preprint},
    date      ={2017},
    eprint    ={https://arxiv.org/abs/1706.06853},
}


\bib{CMNP}{article}{
   author={Corob Cook, Ged},
   author={Moreno, Victor},
   author={Nucinkis, Brita},
   author={Pasini, Federico W.},
   title={On the dimension of classifying spaces for families of abelian
   subgroups},
   journal={Homology Homotopy Appl.},
   volume={19},
   date={2017},
   number={2},
   pages={83--87},
   issn={1532-0073},
   review={\MR{3694094}},
}

\bib{Davbook}{book}{
   author={Davis, Michael W.},
   title={The geometry and topology of Coxeter groups},
   series={London Mathematical Society Monographs Series},
   volume={32},
   publisher={Princeton University Press, Princeton, NJ},
   date={2008},
   pages={xvi+584},
   isbn={978-0-691-13138-2},
   isbn={0-691-13138-4},
   review={\MR{2360474}},
}

 \bib{Degthesis}{thesis}{
     title={Cohomology of split extensions and classifying spaces for families of subgroups},
     author={Degrijse, Dieter},
     type={PhD thesis},
     date={2013},
    }



\bib{DegPe}{article}{
   author={Degrijse, Dieter},
   author={Petrosyan, Nansen},
   title={Bredon cohomological dimensions for groups acting on $\rm
   CAT(0)$-spaces},
   journal={Groups Geom. Dyn.},
   volume={9},
   date={2015},
   number={4},
   pages={1231--1265},
   issn={1661-7207},
   review={\MR{3428413}},
}

\bib{Gui}{collection}{
	label={Cha08},
	title={Guido's book of conjectures},
	series={Monographies de L'Enseignement Math\'ematique},
	volume={40},
	note={A gift to Guido Mislin on the occasion of his retirement from ETHZ
		June 2006;
		Collected by I. Chatterji},
	publisher={L'Enseignement Math\'ematique, Geneva},
	date={2008},
	pages={189},
	isbn={2-940264-07-4},
	review={\MR{2499538}},
}


\bib{HaWis}{article}{
   author={Haglund, F.},
   author={Wise, D. T.},
   title={Special cube complexes},
   journal={Geom. Funct. Anal.},
   volume={17},
   date={2008},
   pages={1551–1620},
   issn={0026-2285},
   review={\MR{2377497}},
}


\bib{HsuWis}{article}{
   author={Hsu, Tim},
   author={Wise, Daniel T.},
   title={On linear and residual properties of graph products},
   journal={Michigan Math. J.},
   volume={46},
   date={1999},
   number={2},
   pages={251--259},
   issn={0026-2285},
   review={\MR{1704150}},
}


   \bib{HuPr}{article}{
   journal={in preparation},
    author={Huang, Jingyin},
   author={Prytu{\l}a, Tomasz},
 
}

   \bib{LM}{article}{
   journal={in preparation},
    author={Leary, Ian J.},
   author={Minasyan, Ashot},
 
}


\bib{Lucze}{article}{
   author={L{\"u}ck, W.},
  title={On the classifying space of the family of virtually cyclic
   subgroups for $\rm CAT(0)$-groups},
   journal={M{\"u}nster J. of Math.},
   volume={2},
   date={2009},
   pages={201--214},

}




\bib{Lusurv}{article}{
	author={L{\"u}ck, W.},
	title={Survey on classifying spaces for families of subgroups},
	conference={
		title={Infinite groups: geometric, combinatorial and dynamical
			aspects},
	},
	book={
		series={Progr. Math.},
		volume={248},
		publisher={Birkh\"auser, Basel},
	},
	date={2005},
	pages={269--322},
	review={\MR{2195456}},
}




\bib{luckbook}{book}{
    AUTHOR = {L{\"u}ck, W.},
     TITLE = {Transformation groups and algebraic {$K$}-theory},
    SERIES = {Lecture Notes in Mathematics},
    VOLUME = {1408},
      NOTE = {Mathematica Gottingensis},
 PUBLISHER = {Springer-Verlag, Berlin},
      YEAR = {1989},
}


\bib{LuMe}{article}{
   author={L\"uck, Wolfgang},
   author={Meintrup, David},
   title={On the universal space for group actions with compact isotropy},
   conference={
      title={Geometry and topology: Aarhus (1998)},
   },
   book={
      series={Contemp. Math.},
      volume={258},
      publisher={Amer. Math. Soc., Providence, RI},
   },
   date={2000},
   pages={293--305},
   review={\MR{1778113}},
}



\bib{LuWe}{article}{
   author={L{\"u}ck, W.},
   author={Weiermann, M.},
   title={On the classifying space of the family of virtually cyclic
   subgroups},
   journal={Pure Appl. Math. Q.},
   volume={8},
   date={2012},
   number={2},
   pages={497--555},
   issn={1558-8599},
   review={\MR{2900176}},
}
  
    \bib{VM}{thesis}{
     type={PhD thesis},
     author={Moreno, Victor},
     status={submitted},
     date={2018},
    }

\bib{OsaPry}{article}{
    title     ={Classifying spaces for families of subgroups for systolic groups},
    author    ={Osajda,Damian},
    author    ={Prytu{\l}a, Tomasz},
    journal={Groups Geom. Dyn.},
    volume={12},
    date={2018},
    number={3},
    pages={1005--1060},
    issn={1661-7207},

}



\bib{Segal}{book}{
   author={Segal, Daniel},
   title={Polycyclic groups},
   series={Cambridge Tracts in Mathematics},
   volume={82},
   publisher={Cambridge University Press, Cambridge},
   date={1983},
   pages={xiv+289},
   isbn={0-521-24146-4},
   review={\MR{713786}},
}
	


\end{biblist}
\end{bibdiv}
\end{document}